\documentclass[12pt,a4paper, reqno]{amsart}
\usepackage[a4paper]{geometry}
\usepackage{amsthm}
\usepackage{amssymb}
\usepackage{amsmath}
\usepackage{mathtools}
\usepackage[utf8]{inputenc}

\usepackage[T1]{fontenc}

\usepackage{wasysym}
\usepackage{color}
\usepackage{enumitem}
\newtheorem{thm}{Theorem}[section]

\newtheorem{lem}[thm]{Lemma}
\newtheorem{prop}[thm]{Proposition}

\theoremstyle{definition}
\newtheorem{defn}[thm]{Definition}
\newtheorem{eks}[thm]{\sc Example}
\theoremstyle{remark}

\numberwithin{equation}{section}
\DeclareMathOperator{\err2}{\mathbb{R}^2}

\DeclareMathOperator{\clco}{\overline{conv}}
\definecolor{bittersweet}{rgb}{1.0, 0.44, 0.37}
\definecolor{byzantine}{rgb}{0.74, 0.2, 0.64}
\title{Daugavet- and Delta-points in absolute sums of Banach spaces}
\author{R. Haller, K. Pirk, and T. Veeorg}
\address{Institute of Mathematics and Statistics, University of Tartu, J.~Liivi 2, 50409 Tartu, Estonia}%
\email{\tiny rainis.haller@ut.ee, katriin.pirk@ut.ee,
triinu.veeorg@gmail.com}%

\keywords{Daugavet property, Daugavet-point, delta-point, absolute sum, diameter two property.}

\begin{document}
\begin{abstract}
    A Daugavet-point (resp.~$\Delta$-point) of a Banach space is a norm one element $x$ for which every point in the unit ball (resp.~element $x$ itself) is in the closed convex hull of unit ball elements that are almost at distance 2 from $x$. A Banach space has the well-known Daugavet property (resp.~diametral local diameter 2 property) if and only if every norm one element is a Daugavet-point (resp.~$\Delta$-point). This paper complements the article ``Delta- and Daugavet-points in Banach spaces'' by T. A. Abrahamsen, R. Haller, V. Lima, and K. Pirk, where the study of the existence of Daugavet- and $\Delta$-points in absolute sums of Banach spaces was started.
\end{abstract}
\subjclass[2010]{Primary 46B20, 46B04}%
\thanks{This work was supported by the Estonian Research Council grants IUT20-57 and PRG877. K.~Pirk was supported in part by University of Tartu Foundation's CWT Estonia travel scholarship.}

\maketitle

\section{Introduction}
Let $X$ be a real Banach space. We denote its closed unit ball by $B_X$, unit sphere by $S_X$, and dual space by $X^*$. 
Following \cite{AHLP} we say that 
\begin{itemize}
    \item[(1)] an $x$ in $S_X$ is a \emph{Daugavet-point} if $B_X \subset \clco \Delta_\varepsilon(x)$ for every $\varepsilon>0$,\smallskip
    \item[(2)] an $x$ in $S_X$ is a \emph{$\Delta$-point} if $x\in \clco\Delta_\varepsilon(x)$ for every $\varepsilon>0$, 
\end{itemize}
    where
$$\Delta_\varepsilon(x) = \{y\in B_X\colon \|x-y\|\geq 2-\varepsilon\}.$$
In the definition of Daugavet- and $\Delta$-point one can equivalently use the set $\{ y\in S_X\colon \|x-y\| \geq 2-\varepsilon\}$ instead of $\Delta_\varepsilon(x)$.

%
The concepts of Daugavet-point and $\Delta$-point originate from the observation in \cite{W} and \cite{BLR_diametral}, on one hand that a Banach space $X$ has the \emph{Daugavet property} (i.e. for every rank-1 bounded linear operator $T\colon X \to X$ we have $\|Id + T\|= 1 + \|T\|$), if and only if every $x\in S_X$ is a Daugavet-point (see \cite{W}); and on the other hand that a Banach space $X$ has the \emph{diametral local diameter 2 property} if and only if every $x\in S_X$ is a $\Delta$-point (see \cite{BLR_diametral}).

In this paper we clarify which absolute sums of Banach spaces have Daugavet- and $\Delta$-points. Recall that a norm $N$ on $\err2$ is called \emph{absolute} if $N(a,b) = N(|a|,|b|)$ for all $(a,b)\in \err2$, and \emph{normalised} if $N(1,0)=N(0,1)=1.$ Given two Banach spaces $X$ and $Y$, and an absolute normalised norm $N$ on $\err2$, we write $X\oplus_N Y$ to denote the direct sum $X\oplus Y$ with the norm $\|(x,y)\|_N = N(\|x\|,\|y\|)$ and we call this Banach space the \emph{absolute sum} of $X$ and $Y$.
Standard examples of absolute normalised norms are the $\ell_p$-norms on $\err2$ for every $p\in [1,\infty]$.

It is known that among all absolute sums only $\ell_1$-sum and $\ell_\infty$-sum can have Daugavet property (see \cite{BKSW}). In contrast, Daugavet- and $\Delta$-points may exist in various absolute sums. Some preliminary results, regarding this matter, were obtained in \cite{AHLP}, let us first recall two of them about Daugavet-points. 
\begin{prop}[see {\cite[Propositions 4.3 and 4.6]{AHLP}}]\label{POH to sum} 
Let $X$ and $Y$ be Banach spaces and $N$ an absolute normalised norm on $\err2$.
\begin{enumerate}
\item[{\rm (a)}] The absolute sum $X\oplus_N Y$ does not have any Daugavet-points, whenever $N$ has the following property:
\medskip
\begin{enumerate}
    \item[$(\alpha)$]{for all nonnegative reals $a$ and $b$ with $N(a,b)=1,$ there exist $\varepsilon>0$ and a neighbourhood $W$ of $(a,b)$ with \[\sup_{(c,d)\in W} c < 1\quad \text{or}\quad \sup_{(c,d)\in W} d < 1,\] such that $(c,d)\in W$ for all nonnegative reals $c$ and $d$ with \[N(c,d)=1\quad \text{and}\quad N\big((a,b)+(c,d)\big) \geq 2-\varepsilon.\]}
\end{enumerate}
\medskip
\item[{\rm (b)}] The absolute sum $X\oplus_N Y$ has Daugavet-points, whenever both $X$ and $Y$ have Daugavet-points, and $N$ has the following property:
\medskip
\begin{enumerate}
    \item[$(\beta)$]\label{beta}{there exist nonnegative reals $a$ and $b$ with $N(a,b)=1$ such that
    \[N\big( (a,b) + (0,1)\big) = 2\quad \text{and}\quad N\big( (a,b) + (1,0)\big)= 2.\]} 
\end{enumerate} 
\smallskip
More precisely, if $x$ and $y$ are Daugavet-points in $X$ and $Y$, respectively, and $(a,b)$ is from $(\beta)$, then $(ax,by)$ is a Daugavet-point in $X\oplus_N Y$.
\end{enumerate}
\end{prop}

It is easy to see that properties $(\alpha)$ and $(\beta)$ exclude each other. However, there are absolute normalised norms that have neither property $(\alpha)$ nor property $(\beta)$ (see discussion in \cite{AHLP}). This motivated us, in part, for further research to clarify the existence of Daugavet-points for all absolute sums (see Section~\ref{sec: from components to sum}).

%
%
%
%
Concerning $\Delta$-points, however, one can easily construct $\Delta$-points in any absolute sum $X\oplus_N Y$, given that both $X$ and $Y$ have $\Delta$-points (see discussion after Lemma 4.1 in \cite{AHLP}). In this paper, our particular aim is to describe the reverse situation, that is the existence of $\Delta$-points in the summands, by assuming that there are $\Delta$-points in the absolute sum.

Our main tools are the following two lemmas that use slices to describe Daugavet- and $\Delta$-points. These lemmas are easily derived from the Hahn--Banach separation theorem.
By a \emph{slice of the unit ball $B_X$} of $X$ we mean a set of the form
$$S(B_X, x^*,\alpha)= \{y\in B_X\colon x^*(y) > 1-\alpha\},$$
where $x^*\in S_{X^*}$ and $\alpha>0$.
\begin{lem}[{\cite[Lemma 2.2]{AHLP}}]\label{Dp-crit}
Let $X$ be a Banach space and $x\in S_X$. Then the following assertions are equivalent:
\begin{itemize}
    \item [{\rm (i)}]{$x$ is a Daugavet-point;}
    \item[{\rm (ii)}]{for every slice $S$ of $B_X$ and every $\varepsilon>0$ there exists $u\in S$ such that $\|x-u\|\ge 2-\varepsilon$.}
\end{itemize}
\end{lem}
\begin{lem}[{\cite[Lemma 2.1]{AHLP}}]\label{Delta-points crit}
Let $X$ be a Banach space and $x\in S_X$. Then the following assertions are equivalent:
\begin{itemize}
    \item [{\rm (i)}]{$x$ is a $\Delta$-point;}
    \item[{\rm (ii)}]{for every slice $S$ of $B_X$, with $x\in S$, and every $\varepsilon>0$ there exists $u\in S$ such that $\|x-u\|\ge 2-\varepsilon$.}
\end{itemize}
\end{lem}

In the following, we present our main results regarding both Daugavet- and $\Delta$-points. In Section \ref{sec: from components to sum}, we generalise part (b) of Proposition \ref{POH to sum} to all absolute normalised norms that do not have property ($\alpha$), which completes the research about the existence of Daugavet-points in the absolute sum, given that component spaces have Daugavet-points (see Theorem \ref{A-OH to direct sum}).

In Section \ref{sec: from sum to components}, we consider the existence of Daugavet-points in component spaces, assuming that the absolute sum has Daugavet-points. In this case, we prove that at least one component space has Daugavet-points (see Theorems~\ref{Daugavet-points to components N} and \ref{Daugavet-points to components infty}).

In Section \ref{sec: Delta-points}, we deal with the existence of $\Delta$-points in component spaces of an absolute sum with $\Delta$-points. For absolute normalised norms that differ from $\ell_\infty$-norm the results are similar to the case of Daugavet-points, i.e. at least one of component spaces has $\Delta$-points (see Theorem \ref{Delta-points on components}). However, the point $(x,y)$ with $\|x\|=\|y\|=1$ in the absolute sum equipped with $\ell_\infty$-norm can surprisingly be a $\Delta$-point even if neither $x$ nor $y$ is a $\Delta$-point in the respective component space (see Example \ref{Delta-point ex} and Proposition \ref{Unexpected Delta-points}). 
\section{Daugavet-points from summands to absolute sums}\label{sec: from components to sum}

In this section we focus on absolute sums inheriting Daugavet-points from component spaces. As pointed out in Introduction, the absolute sum with a norm $N$ satisfying property ($\alpha$), does not have any Daugavet-points (see  part (a) of Proposition \ref{POH to sum}). However, if both component spaces have Daugavet-points, then the absolute sum equipped with a norm $N$ satisfying property ($\beta$) also has Daugavet-points (see part (b) of Proposition \ref{POH to sum}). We complete this direction by specifying the situation for all the other absolute normalised norms besides the ones with properties ($\alpha$) or ($\beta$).

\begin{defn}\label{$A$-OH def}
Let $X$ be a Banach space and $A\subset S_X$. We say that the norm on $X$ is \emph{A-octahedral} ($A$-OH) if for every $x_1,\dots, x_n \in A$ and every $\varepsilon>0$ there exists $y\in S_X$ such that $\|x_i+y\|\geq 2-\varepsilon$ for every $i\in \{1,\dots,n\}.$
\end{defn}
It is evident that the octahedrality of a norm in its usual sense (see \cite{G} and \cite[Proposition 2.4]{HLP}) means that the norm is $S_X$-OH. 
We use this more general term, $A$-octahedrality to describe absolute sums, or more precisely absolute normalised norms, for which the absolute sums possess Daugavet-points.
The name, $A$-octahedrality, is justified by the fact that property $(\beta)$ has also been called positive octahedrality in \cite{HLN}.
Note that absolute normalised norm $N$ on $\err2$ with property $(\beta)$ is exactly $\{(0,1),(1,0)\}$-OH.
%

Consider an absolute normalised norm $N$ on $\err2$. We now define a specific set $A$ that is considered from here on until the end of Section \ref{sec: from sum to components}. Set
\begin{equation}\label{A-OH,cd}
    c=\max_{N(e,1)=1}e\quad \textnormal{and}\quad d=\max_{N(1,f)=1}f,\tag{$*$} 
\end{equation}
and define
$$A= \{(c,1),(1,d)\}.$$

Suppose that the norm $N$ is $A$-OH. By Definition \ref{$A$-OH def} there exists $(a,b)\in \err2$ with the following property:
\begin{equation}\tag{$**$}
\begin{split}
&a,b\geq 0,\quad N(a,b)=1,\quad \text{and}\\
&N\big((a,b)+(c,1)\big)=2\quad  \text{and}\quad  N\big((a,b)+(1,d)\big)=2.\label{A-OH, N=2}
\end{split}
\end{equation}

It is easy to see that, firstly, as mentioned above, the norms with property ($\beta$) are $A$-OH (for the specified $A$ as well), and secondly, $A$-OH norms do not have  property $(\alpha)$. Moreover, part (b) of Proposition \ref{POH to sum} can be extended for $A$-OH norms $N$ as well.
\begin{thm}\label{A-OH to direct sum}
Let $X$ and $Y$ be Banach spaces, $x\in S_X$, $y\in S_Y$, and let $N$ be an $A$-OH norm with $(a,b)$ as in $(**)$. If $x$ and $y$ are Daugavet-points in $X$ and $Y$, respectively, then $(ax,by)$ is a Daugavet-point in $X\oplus_N Y$.
\end{thm}
\begin{proof}
Assume that $x$ and $y$ are Daugavet-points. Set $Z=X\oplus_NY$ and
fix $f=(x^*,y^*)\in S_{Z^*}$, $\alpha>0$, and $\varepsilon>0$. Choose $\delta>0$ to satisfy $\delta N(1,1)<\varepsilon$. By Lemma \ref{Dp-crit} we obtain $u\in B_X$ and $v\in B_Y$ such that
$$x^*(u)\ge\Big(1-\frac{\alpha}{2}\Big)\|x^*\|\quad \textnormal{and}\quad y^*(v)\ge\Big(1-\frac{\alpha}{2}\Big)\|y^*\|$$
and
$$\|x-u\|\ge 2-\delta\quad \textnormal{and}\quad \|y-v\|\ge 2-\delta.$$ 

By the properties of absolute normalised norms (see \cite{H}, p 317) and $A$-OH norms, there exist $k,l\ge 0$ such that
$$N(k,l)=1,\quad N\big((a,b)+(k,l)\big)=2,\quad\textnormal{and}\quad k\|x^*\| + l\|y^*\|=1.$$
Therefore $(ku,lv)\in S(B_Z,f,\alpha)$, because
$$f(ku,lv) = kx^*(u)+ly^*(v) \ge\Big(1-\frac{\alpha}{2}\Big)(k\|x^*\|+l\|y^*\|)>1-\alpha.$$
On the other hand, from $\|x-u\|\geq 2-\delta$ and $\|y-v\| \geq 2-\delta$, we get that 
$$\big\Vert ax-ku\big\Vert \geq a+k - \delta\quad
\text{and}\quad
\big\Vert by-lv\big\Vert \geq b+l-\delta.$$
In conclusion we have that
\begin{align*}
\|(ax,by)-(ku,lv)\|_N&=N(\|ax-ku\|,\|by-lv\|)\\
&\ge N(a+k-\delta, b+l-\delta)\\
&\ge N(a+k,b+l)-N(\delta, \delta)\\
&= N\big((a,b)+(k,l)\big)-\delta N(1,1)\\
&> 2-\varepsilon,
\end{align*}
which means that $(ku,lv)$ satisfies the necessary conditions such that according to Lemma \ref{Dp-crit} $(ax,by)$ is a Daugavet-point.
\end{proof}
In order to have Daugavet-points in the absolute sum it is enough, for some $A$-OH norms, to assume that only one summand has a Daugavet-point. The following two propositions describe these special occasions, we drop the proofs since they are similar to the previous one.
\begin{prop}\label{(x,0)and(0,y)Daugavet-points}
Let $X$ and $Y$ be Banach spaces, $x\in S_X$ and $y\in S_Y$, and let $N$ be an $A$-OH norm with $(a,b)$ as in $(**)$.
\begin{enumerate}
    \item[{\rm(a)}] If $b=0$ and $x$ is a Daugavet-point in $X$, then $(ax,by)=(x,0)$ is a Daugavet-point in $X\oplus_NY$.
    \item[{\rm(b)}]  If $a=0$ and $y$ is a Daugavet-point in $Y$, then $(ax,by)=(0,y)$ is a Daugavet-point in $X\oplus_NY$.
\end{enumerate}
\end{prop}
\begin{prop}\label{(x,by)Daugavet-point in infty-sum}
Let $X$ and $Y$ be Banach spaces, $x\in S_X$ and $y\in S_Y$. \begin{enumerate}
    \item[{\rm(a)}] If $x$ is a Daugavet-point in $X$, then $(x,by)$ is a Daugavet-point in $X\oplus_{\infty}Y$ for every $b\in [0,1]$. 
    \item[{\rm(b)}] If $y$ is a Daugavet-point in $Y$, then $(ax,y)$ is a Daugavet-point in $X\oplus_{\infty}Y$ for every $a\in [0,1]$.
\end{enumerate}
\end{prop}
 With this we have widened the class of absolute sums inheriting Daugavet-points from the component spaces (see also the discussion after Remark 5.9 in \cite{AHLP}). At this point, it is still not clear, though, whether all absolute normalised norms are covered, since there could be other absolute normalised norms 'in between' the norms with property ($\alpha$) and $A$-OH norms. We will prove now that actually this is not the case, all absolute normalised norms are either with property ($\alpha$) or $A$-OH.

It is not hard to see that in the definition of property ($\alpha$), we can relax the condition $\varepsilon >0$ and let $\varepsilon=0$. The corresponding reformulation of property ($\alpha$) is the following:
\begin{enumerate}
    \item[$(\alpha)$]{for all nonnegative reals $a$ and $b$ with $N(a,b)=1,$ there exists a neighbourhood $W$ of $(a,b)$ with \[\sup_{(c,d)\in W} c < 1\quad \text{or}\quad \sup_{(c,d)\in W} d < 1,\] such that $(c,d)\in W$ for all nonnegative reals $c$ and $d$ with \[N(c,d)=1\quad \text{and}\quad N\big((a,b)+(c,d)\big) = 2.\]}
\end{enumerate}
\begin{prop}\label{A-OH eqv neg beta}
Every absolute normalised norm on $\mathbb R^2$ is either with property $(\alpha)$ or $A$-OH. 
\end{prop}
\begin{proof}
Let $N$ be an absolute normalised norm that does not have property $(\alpha)$. Then there exist $a,b\geq 0$ with $N(a,b)=1$ such that for every $(a,b)$-neighbourhood $W$ which satisfies either $\sup_{(c,d)\in W} c < 1$ or $\sup_{(c,d)\in W} d < 1,$ there exist $c,d\geq 0$ with $N(c,d)=1$ such that $(c,d)\notin W$ and $N\big((a,b)+(c,d)\big) = 2.$
To show that $N$ is $A$-OH, we need to find $c,d\geq 0$ satisfying
$$N(c,1)=N(1,d)=1\quad \textnormal{and}\quad N\big((a,b)+(c,1)\big)=N\big((a,b)+(1,d)\big)=2.$$

If $a=1$ ($b=1$, respectively), then take $d=b$ ($c=a$, respectively). However, if $a\neq 1$, then for every $n\in \mathbb{N}$ large enough $(a < 1-1/n),$ by taking $W_n = \{(x,y)\colon x\leq 1-1/n\},$ we can find $c_n,d_n\geq 0$ with $N(c_n,d_n)=1$ such that $(c_n,d_n)\notin W_n$ and $N\big((a,b)+(c_n,d_n)\big) = 2.$ Passing to a subsequence if necessary, we can assume that $(c_{n},d_{n})\to (1,d)$ for some $d\geq 0$. Obviously $N(1,d)=1$ and $N\big((a,b) +(1,d)\big) = 2.$ It can be proved similarly that if $b\neq 1$, then there exists $c\geq 0$ with $N(c,1)=1$ and $N\big((a,b) +(c,1)\big) = 2.$ Combining these facts we have that $N$ is $A$-OH.
\end{proof}
\section{Daugavet-points from absolute sums to summands}\label{sec: from sum to components}
In \cite{AHLP} the authors stated a question about the existence of Daugavet-points in Banach spaces $X$ and $Y$, given that the absolute sum $X\oplus_N Y$ has Daugavet-points (see \cite[Problem 1]{AHLP}). In this section we give an answer to that question. Recall that this question is open only for $A$-OH norms $N$, whereas for norms $N$ with property $(\alpha)$ we do not have any Daugavet points in $X\oplus_N Y$.

\begin{thm}\label{Daugavet-points to components N}
Let $X$ and $Y$ be Banach spaces, $x\in B_X$, $y\in B_Y$, and let $N$ be an absolute normalised norm on $\mathbb R^2$, different from $\ell_\infty$-norm. Assume that $(x,y)$ is a Daugavet-point in $X\oplus_N Y$.
\begin{enumerate}
    \item[{\rm(a)}] If $x\not =0$, then $x/\|x\|$ is a Daugavet-point in $X$.
    \item[{\rm(b)}] If $y\not =0$, then $y/\|y\|$ is a Daugavet-point in $Y$.
\end{enumerate}
\end{thm}

\begin{proof} 
We prove only the first statement, the second can be proved similarly.  Suppose that $x\not =0$. Fix $x^* \in S_{X^*}$, $\alpha>0$, and $\varepsilon>0$. We will find $u\in S(B_X,x^*,\alpha)$ such that $\big\|x/\|x\|-u\big\|\ge 2-\varepsilon$. Set $f=(x^*,0)$ and $Z=X\oplus_N Y$. Then $f\in S_{Z^*}$. Choose $\delta>0$ such that, for every $p,q,r\ge0$, if
$$2-\delta \le N(p,q)\leq N(r,q) \leq 2\quad \textnormal{and}\quad  q<2-\delta,$$
then $|p-r|<\|x\|\varepsilon/2$.
There is no loss of generality in assuming that $\delta\le \varepsilon/2$, $\delta\le \alpha$, and $(1-\delta)N(1,1)>1$ (here we use the fact that $N(1,1)>1$, i.e. $N$ is not $\ell_\infty$-norm).

Since $(x,y)$ is a Daugavet-point in $Z$, there exists $(u,v)\in S(B_Z, f, \delta)$ such that $\|(x,y)-(u,v)\|_N \ge 2-\delta$. Consequently, 

$$x^*(u) = f(u,v) > 1-\delta \geq 1-\alpha,$$
which gives us that $u\in S(B_X,x^*,\alpha)$ and $\|u\|>1-\delta$. We also conclude that $\|v\|<1-\delta,$ because otherwise
$$N(\|u\|, \|v\|) \geq N(1-\delta,1-\delta) = (1-\delta)N(1,1)>1,$$
a contradiction.
In addition we have that
$$2-\delta \le N\big(\|x-u\|,\|y -v\|\big) \leq N\big(\|x\| + \|u\|,\|y-v\|\big) \leq 2$$
and
$$\|y -v\|\leq\|y\|+\|v\|< 1 + 1-\delta = 2-\delta.$$
Hence, by the choice of $\delta$, we have that
$$\big| \|x-u\| - (\|x\|+\|u\|)\big| < \|x\|\varepsilon/2.$$
Thus $\|x-u\|>\|x\| + \|u\|-\|x\|\varepsilon/2,$ and therefore,
\begin{align*}
    \Big\|\frac{x}{\|x\|}-u\Big\| & = \Big\|\frac{1}{\|x\|}(x-u)-\Big(u - \frac{1}{\|x\|}u\Big)\Big\| \\&
    \geq
    \frac{1}{\|x\|}\|x-u\| - \Big(\frac{1}{\|x\|} - 1\Big)\|u\|\\&
    \geq
    \frac{1}{\|x\|}\Big(\|x\| + \|u\| - \frac{\|x\|\varepsilon}{2}\Big) - \frac{\|u\|}{\|x\|}+\|u\|\\&
    =
    1+\|u\| - \frac{\varepsilon}{2} \\& >
    1+1-\delta - \frac{\varepsilon}{2}\\& \geq
     2-\varepsilon.
\end{align*}
According to Lemma \ref{Dp-crit}, the element $x/\|x\|$ is a Daugavet-point in $X$.
\end{proof}

%
%
%


Let us now move on to the case of $\ell_\infty$-norm. Recall that if $x$ is a Daugavet-point in $X$, then $(x,y)$ is a Daugavet-point in $X\oplus_\infty Y$ for every $y\in B_Y$ (see Proposition \ref{(x,by)Daugavet-point in infty-sum}). This means that if $\ell_\infty$-sum of two Banach spaces has a Daugavet-point, then both summands need not have a Daugavet-point. However, in the following we prove that at least one of the summands has a Daugavet-point.
\begin{thm}\label{Daugavet-points to components infty}
Let $X$ and $Y$ be Banach spaces, $x\in B_X$, $y\in B_Y$. Assume that $(x,y)$ is a Daugavet-point in $X\oplus_\infty Y$. Then $x$ is a Daugavet-point in $X$ or $y$ is a Daugavet-point in $Y$.
\end{thm}
\begin{proof}
Firstly, look at the case, where only one of $x$ and $y$ has norm 1, and the other has norm less than 1. 
Assume that $\|y\|<1$ and let us prove that $x$ is a Daugavet-point in $X$. (The statement, if $\|x\|<1$, then $y$ is a Daugavet-point in $Y$, can be proved similarly.) Choose $\delta>0$ such that $\delta\le \varepsilon$ and $\|y\| < 1-\delta.$ 
%
%
%
Fix $x^*\in S_{X^*}$, $\alpha>0$ and $\varepsilon>0$. Set $Z=X\oplus_\infty Y$ and
$f=(x^*,0)\in S_{Z^*}$. Since $(x,y)$ is a Daugavet-point then there exists $(u,v)\in S(B_Z, f, \alpha)$ such that $\|(x,y)-(u,v)\|_{\infty}\ge 2-\delta$. Therefore, $x^*(u)=f(u,v)>1-\alpha$, i.e. $u\in S(B_X, x^*, \alpha)$ and
$$\|y-v\|\le \|y\|+\|v\| <1-\delta+1=2-\delta.$$

Combining this with the fact that
$$\|(x,y) - (u,v)\|_\infty =
    \max\{\|x-u\|,\|y-v\|\} \geq 2-\delta,$$
we get that $\|x-u\|\ge 2-\delta\geq 2-\varepsilon$. Thus, $x$ is a Daugavet-point in $X$.

Secondly, consider the case, where both $x$ and $y$ are of norm 1, and neither of them is a Daugavet-point. 
Then we can fix slices $S(B_X, x^*,\alpha)$ and $S(B_Y,y^*,\alpha)$, and $\varepsilon>0$ such that $S(B_X, x^*,\alpha)\cap \Delta_{\varepsilon}(x) = \emptyset$ and $S(B_Y, y^*,\alpha)\cap \Delta_{\varepsilon}(y) = \emptyset.$ There is no loss of generality in assuming that $\alpha < \varepsilon <1$.
Set $f= 1/2(x^*,y^*)$ and $Z=X\oplus_\infty Y$, and consider the slice $S(B_Z, f, \alpha/2).$ Note that
$$S(B_Z, f, \alpha/2)\subset S(B_X,x^*,\alpha) \times S(B_Y,y^*,\alpha).$$
%
Let $(u,v)\in S(B_Z, f, \alpha/2)\cap S_Z$ be arbitrary. Then $$\|u\|>1-\alpha>0\quad \textnormal{and}\quad \|v\|>1-\alpha>0,$$
and
$$u/\|u\| \in S(B_X, x^*,\alpha)\quad \textnormal{and}\quad v/\|v\|\in S(B_Y,y^*,\beta).$$
Therefore
\begin{align*}
    \|(u,v) - (x,y)\|_\infty & = \max\{\|u-x\|,\|v-y\|\}\\& \leq
    \max\Big\{\Big\Vert u-\frac{u}{\|u\|}\Big\Vert + \Big\Vert \frac{u}{\|u\|} -x\Big\Vert, \Big\Vert v-\frac{v}{\|v\|}\Big\Vert + \Big\Vert \frac{v}{\|v\|} -y\Big\Vert\Big\} \\& <
    \alpha + 2-\varepsilon \\&
    = 2-(\varepsilon - \alpha).    
\end{align*}
As a result, $S(B_Z, f, \alpha/2)\cap \Delta_{\varepsilon-\alpha} (x,y)=\emptyset,$ which by Lemma \ref{Dp-crit} implies that $(x,y)$ is not a Daugavet-point.
\end{proof}
\section{Delta-points from direct sums to summands}\label{sec: Delta-points}
%
%
%
As mentioned in Introduction, $\Delta$-points pass from component spaces to the absolute sum for every absolute normalised norm. In fact, this observation let the authors of \cite{AHLP} conclude that $\Delta$-points are indeed different from Daugavet-points. In this section we clarify the existence of $\Delta$-points in the component spaces, given that the absolute sum has $\Delta$-points. Surprisingly, the case of $\Delta$-points is different from the case of Daugavet-points even for $\ell_\infty$-norm (see Proposition \ref{Unexpected Delta-points}).

From here on, we consider arbitrary absolute normalised norms $N$. Firstly, we show that, as expected, for most absolute normalised norms $N$ the component spaces have $\Delta$-points, given the absolute sum has $\Delta$-points.
\begin{thm}\label{Delta-points on components}
Let $X$ and $Y$ be Banach spaces, $x\in S_X$, $y\in S_Y$, $N$ an absolute normalised norm on $\err2$, and $a,b\geq 0$ such that $N(a,b)=1$. Assume that $(ax,by)$ is a $\Delta$-point in $X\oplus_N Y$.
\begin{enumerate}
    \item[{\rm(a)}] If $b\neq 1$, then $x$ is a $\Delta$-point in $X$.
    \item[{\rm(b)}] If $a\neq 1$, then $y$ is a $\Delta$-point in $Y$.
\end{enumerate}
\end{thm}
\begin{proof}
We prove only the first statement, the second can be proved similarly. Assume that $b\not =1$. Note that then $a\neq 0$. Let $c,d\geq 0$ be such that $N^*(c,d)=1$ and $ac+bd=1$.

Suppose that $x$ is not a $\Delta$-point in $X$. Then, by Lemma \ref{Delta-points crit}, there exist $x^*\in S_{X^*}$, $\alpha>0$, and $\varepsilon>0$ such that
$$x\in S(B_X, x^*,\alpha) \quad \textnormal{and}\quad S(B_X, x^*,\alpha)\cap \Delta_\varepsilon(x) = \emptyset.$$
Let $y^*\in S_{Y^*}$ be such that $y^*(y) = 1$ and let $f=(cx^*,(1-\alpha)dy^*)$. Then 
$$f(ax,by)=acx^*(x)+(1-\alpha)bdy^*(y)>(1-\alpha)(ac+bd)=1-\alpha.$$ 
Choose $\beta,\gamma>0$ such that $\beta<a\varepsilon$ and $\beta<\gamma \varepsilon$, and $$f(ax,by) > 1 - (\alpha-\gamma).$$
Now choose $\delta>0$ such that, for every $p,q,r\ge0$, if
$$2-\delta \le N(p,q)\leq N(r,q) \leq 2 \quad \textnormal{and}\quad  q<2-\delta,$$
then $|p-r|<\beta$.
There is no loss of generality in assuming that $b<1-\delta$.
There exists $(u,v)\in B_Z$, where $Z=X\oplus_N Y$, such that
$$f(u,v) > 1 - (\alpha-\gamma)\quad \textnormal{and}\quad \|(ax,by)-(u,v)\|_N\ge 2-\delta.$$
Then
\begin{align*}
    cx^*(u)+(1-\alpha)d\|v\|&\ge cx^*(u)+(1-\alpha)dy^*(v)\\&
    =f(u,v)\\&
    >1-(\alpha-\gamma)\\
    &>1-\alpha\\&
    \ge (1-\alpha)(c\|u\|+d\|v\|),
\end{align*}
which yields
$$cx^*(u) > (1-\alpha) c\|u\|,$$
i.e. $x^*(u/\|u\|) > 1-\alpha$. Since $S(B_X,x^*,\alpha)\cap \Delta_\varepsilon(x) = \emptyset,$ we know now that $\big\|x - u/\|u\|\big \| < 2-\varepsilon.$
We now show that $\|ax-u\| < a+\|u\|-\beta$. Let us consider two cases. If $\|u\|\geq a$, then
\begin{align*}
    \|ax-u\| & \leq \Big\Vert ax - a\frac{u}{\|u\|}\Big\Vert + \Big\Vert a\frac{u}{\|u\|} - u\Big\Vert\\& \leq
    a(2-\varepsilon) + \big\vert a-\|u\|\big\vert\\&
    = a+\|u\| -a\varepsilon\\&
    <a+\|u\| -\beta.
\end{align*}
On the other hand, if $a\geq \|u\|$, we have
\begin{align*}
c\|u\|+(1-\alpha)d\|v\|&\ge cx^*(u)+(1-\alpha)dy^*(v)\\&
=f(u,v)\\
&>1-\alpha+\gamma\\&
\ge (1-\alpha)d\|v\|+\gamma,
\end{align*}
from which we conclude $\|u\|\ge c\|u\|> \gamma$. Now we see that
\begin{align*}
    \|ax-u\| & \leq \big\Vert ax - \|u\|x\big\Vert + \big\Vert \|u\|x - u\big\Vert\\& \leq
    a- \|u\| +\|u\|(2-\varepsilon)\\&
    = a+\|u\| -\|u\|\varepsilon\\& < a+\|u\|-\gamma\varepsilon\\&<a+\|u\| -\beta.
\end{align*}
That gives us the following:
\begin{align*}
    2-\delta \le \|(ax,by)-(u,v)\|_N &= N(\|ax-u\|,\|by-v\|)\\&
    \le N(a+\|u\|-\beta,b+\|v\|)
\end{align*}
and therefore
$$2-\delta\le N(a+\|u\|-\beta,b+\|v\|) \le N(a+\|u\|,b+\|v\|)\le 2.$$
Since $b+\|v\|< 2-\delta$, we have by the choice of $\delta$ that $\big|(a+\|u\|-\beta)-(a+\|u\|)\big|<\beta$, i.e. $\beta<\beta$, a contradiction. Hence $x$ is a $\Delta$-point in $X$.
\end{proof}
Theorem \ref{Delta-points on components} does not cover the case $a=b=1$ (for $\ell_\infty$-norm). Our original assumption was that if $(x,y)$ is a $\Delta$-point in $X\oplus_\infty Y$ for some $x\in S_X$ and $y\in S_Y$, then either $x$ or $y$ must also be a $\Delta$-point (respectively in $X$ or $Y$), similarly to the case of Daugavet-points. However, we show that in this case our intuition was wrong. Moreover, we introduce the conditions that $x\in S_X$ and $y\in S_Y$ must satisfy in order for $(x,y)$ to be a $\Delta$-point in $X\oplus_\infty Y$. These results rely heavily on the concept of another type of unit sphere elements similar to $\Delta$-points (compare with Lemma \ref{Delta-points crit}).
\begin{defn}
Let $X$ be a Banach space, $x\in S_X$, and $k>1$. We say that $x$ is a \emph{$\Delta_k$-point} in $X$, if for every $S(B_X,x^*,\alpha)$ with $x\in S(B_X,x^*,\alpha)$ and for every $\varepsilon>0$ there exists $u\in S(B_X,x^*,k\alpha)$ such that $\|x-u\|\ge 2-\varepsilon$.
\end{defn}
Every $\Delta$-point is obviously a $\Delta_k$-point for every $k>1$. In contrast, the reverse does not hold, since the upcoming example shows the existence of a $\Delta_k$-point that is not a $\Delta$-point, which proves that the concepts of $\Delta$-point and $\Delta_k$-point do not coincide.
\begin{eks}\label{Delta-point ex}
Let $X$ and $Y$ be Banach spaces, $x\in S_X$ and $y\in S_Y$, and let $k>1$. Set $Z=X\oplus_1Y$ and $z=\big((1-1/k)x ,y/k\big)$. Assume that $x$ is not a $\Delta$-point in $X$ and $y$ is a $\Delta$-point in $Y$. Then, according to Theorem \ref{Delta-points on components}, $z$ is not a $\Delta$-point in $Z$.

Fix $f=(x^*,y^*)\in S_{Z^*}$ and $\alpha>0$, such that $f(z)>1-\alpha$, and fix $\varepsilon>0$. Then 
$$1-\frac{1}{k}+\frac{1}{k}y^*(y)\ge \Big(1-\frac{1}{k}\Big) x^*(x)+ \frac{1}{k}y^*(y)= f(z) >1-\alpha.$$
It follows that $y^*(y)>1-\alpha k$. Since $y$ is a $\Delta$-point, there exists $v\in B_Y$ such that
$$y^*(v)>1-\alpha k\quad \textnormal{and} \quad \|y-v\|\ge 2-\varepsilon.$$
Then $f(0,v)=y^*(v)>1-\alpha k$,
i.e. $(0,v)\in S(B_Z,f,\alpha k)$, and
\begin{align*}
\Big\|\Big(\Big(1-\frac{1}{k}\Big)x, \frac{1}{k} y\Big)-(0,v)\Big\|_1 &= \Big(1-\frac{1}{k}\Big)\|x\|+\Big\|\frac{1}{k}y-v\Big\|\\
&\ge \Big(1-\frac{1}{k}\Big)+ \|y-v\|-\Big(1-\frac{1}{k}\Big)\|y\|\\& \ge 2-\varepsilon.
\end{align*}
This proves that $z$ is a $\Delta_k$-point.
\end{eks}
Surprisingly $(x,y)$ with $\|x\|=\|y\|=1$ can be a $\Delta$-point in $X\oplus_\infty Y$ even if neither $x$ nor $y$ is a $\Delta$-point in $X$ and $Y$, respectively. This is immediate from Example \ref{Delta-point ex} and following Proposition \ref{Unexpected Delta-points}.
\begin{prop}\label{Unexpected Delta-points}
Let $X$ and $Y$ be Banach spaces, $x\in S_X$ and $y\in S_Y$. Let $p,q>1$ satisfy $1/p+1/q=1$.
\begin{itemize}
    \item [{\rm (a)}]{If $x$ is a $\Delta_p$-point in $X$ and $y$ is a $\Delta_q$-point in $Y$, then $(x,y)$ is a $\Delta$-point in $X\oplus_\infty Y$.}
    \item[{\rm (b)}]{If $x$ is not a $\Delta_p$-point in $X$ and $y$ is not a $\Delta_q$-point in $Y$, then $(x,y)$ is not a $\Delta$-point in $X\oplus_\infty Y$.}
\end{itemize}
\end{prop}
In fact, Proposition \ref{Unexpected Delta-points} gives equivalent condition for $(x,y)$ being a $\Delta$-point in $X\oplus_\infty Y$ where neither $x$ nor $y$ is a $\Delta$-point in $X$ and $Y$, respectively (see Proposition \ref{nec-suf cond for Delta-p in infty-sum} below).  
%
%
\begin{proof}[Proof of Proposition~\ref{Unexpected Delta-points}]
(a) Assume that $x$ is a $\Delta_p$-point in $X$ and $y$ is $\Delta_q$-point in $Y$.
Set $Z=X\oplus_\infty Y$. Fix $f=(x^*,y^*)\in S_{Z^*}$ and $\alpha>0$ such that $(x,y)\in S(B_Z,f,\alpha)$, and fix $\varepsilon>0$. Then
\[
x^*(x)+y^*(y)=f(x,y)>1-\alpha
\]
from what we get
\[
x^*(x)>1-(\alpha+y^*(y))=\|x^*\|-\big(\alpha+y^*(y)-\|y^*\|\big).
\]
We now show that there exists $(u,v)\in S(B_Z,f,\alpha)$ such that $$\|(x,y)-(u,v)\|_\infty \geq 2-\varepsilon.$$
Let us consider two cases. If $\alpha+y^*(y)-\|y^*\|\le \alpha/p$, then $x^*(x)>\|x^*\|-\alpha/p$ and therefore, since $x$ is a $\Delta_p$-point, there exists $u\in B_X$ such that $x^*(u)>\|x^*\|-\alpha$ and $\|x-u\|\ge 2-\varepsilon$. Let $v\in B_Y$ be such that 
\[
f(u,v)= x^*(u)+y^*(v)>\|x^*\|-\alpha+\|y^*\|=1-\alpha.
\]
Then also $\|(x,y)-(u,v)\|_\infty=\max\{\|x-u\|,\|y-v\|\}\ge 2-\varepsilon$.

If $\alpha+y^*(y)-\|y^*\|> \alpha/p$, then
\[
y^*(y)>\|y^*\|- \alpha+\frac{\alpha}{p}=\|y^*\|-\frac{1}{q}\alpha
\]
and analogically, using the fact that $y$ is a $\Delta_q$-point, we can find $(u,v)\in S(B_Z,f,\alpha)$ such that $\|(x,y)-(u,v)\|_\infty \ge 2-\varepsilon$. Therefore $(x,y)$ is a $\Delta$-point.

(b) Assume that $x$ is not a $\Delta_p$-point in $X$ and $y$ is not a $\Delta_q$-point in $Y$.
By definition there exist $x^*\in S_{X^*}$, $y^*\in S_{Y^*}$, and $\alpha_1, \alpha_2,\varepsilon>0$, with $x\in S(B_X,x^*,\alpha_1)$ and $y\in S(B_X,y^*,\alpha_2)$ such that for every $u\in S(B_X,x^*,p\alpha_1)$ and for every $v\in S(B_Y,y^*,q\alpha_2)$ we have
$$\|x-u\|< 2-\varepsilon\quad \textnormal{and}\quad \|y-v\|< 2-\varepsilon.$$
Set $Z=X\oplus_\infty Y$. Let $\lambda\in (0,1)$ satisfy $(1-\lambda)/\lambda=(p\alpha_1)/(q\alpha_2)$, let $\alpha=\lambda \alpha_1+(1-\lambda) \alpha_2$ and let $f=(\lambda x^*, (1-\lambda)y^*)\in S_{Z^*}.$ Then
\begin{align*}
f(x,y)&=\lambda x^*(x)+(1-\lambda)y^*(y)\\&
>\lambda(1-\alpha_1)+(1-\lambda)(1-\alpha_2)\\&
=1-\alpha.
\end{align*}
Fix $(u,v)\in S(B_Z,f,\alpha)$. 
From
$$1-\alpha < f(u,v) = \lambda x^*(u)+(1-\lambda) y^*(v) \leq \lambda x^*(u) + 1-\lambda$$
we get that
\begin{align*}
x^*(u)& > 1-\frac{\alpha}{\lambda}=1- \Big(\alpha_1+\frac{1-\lambda}{\lambda}\alpha_2\Big)=1-p\Big(\frac{\alpha_1}{p}+\frac{\alpha_1}{q}\Big)=1-p\alpha_1.
\end{align*}
Therefore $u\in S(B_X,x^*,p\alpha_1)$ and analogically $v\in S(B_Y,y^*,q\alpha_2)$. It follows that $\|x-u\|< 2-\varepsilon$ and $\|y-v\|< 2-\varepsilon$. Consequently
$$\|(x,y)-(u,v)\|_\infty=\max\{\|x-u\|,\|y-v\|\}< 2-\varepsilon$$
and thus, $(x,y)$ is not a $\Delta$-point.
%
%
%
%
%
%
%
%
%
\end{proof}
\begin{prop}\label{nec-suf cond for Delta-p in infty-sum}
Let $X$ and $Y$ be Banach spaces and $x\in S_X$ and $y\in S_Y$. Assume that neither $x$ nor $y$ is a $\Delta$-point in $X$ and $Y$, respectively. Then the following statements are equivalent:

\begin{enumerate}
    \item[{\rm(i)}] there exist $p,q>1$ with $1/p+1/q=1$ such that $x$ is $\Delta_p$-point in $X$ and $y$ is $\Delta_q$-point in $Y$;
    
    \item[{\rm(ii)}] for every $p,q>1$ with $1/p+1/q=1$ either $x$ is $\Delta_p$-point in $X$ or $y$ is $\Delta_q$-point in $Y$.
\end{enumerate}
\end{prop}

\begin{proof}
(i) $\Rightarrow$ (ii). Assume that (i) holds. Let $p,q>1$ be such that $1/p+1/q=1$. According to $(a)$ $x$ is $\Delta_{p'}$-point in $X$ and $y$ is $\Delta_{q'}$-point in $Y$ for some $p',q'>1$ with $1/p'+1/q'=1$. Then $p'\ge p$ or $q'\ge q$ and therefore $x$ is $\Delta_{p}$-point in $X$ or $y$ is $\Delta_{q}$-point in $Y$, hence (ii) holds.

(ii) $\Rightarrow$ (i). Assume that (ii) holds. Define \[A=\{k\in [1,\infty) \colon x \text{ is $\Delta_k$-point in $X$}\}\] and \[B=\{k\in [1,\infty) \colon y \textnormal{ is $\Delta_k$-point in $Y$}\}.\] 

Firstly, let us examine the case where set $A$ is nonempty. Let $a=\inf A$. We show that $a\in A$. 
Fix $x^*\in S_{X^*}$, $\alpha>0$ and $\varepsilon>0$ such that $x\in S(B_X,x^*,\alpha)$. Let $\gamma>0$ be such that $x^*(x)>1-(\alpha-\gamma)$ and let $k=a\alpha/(\alpha-\gamma)$. Then $k>a$ and therefore $k\in A$. Since $x\in S(B_X,x^*,\alpha-\gamma)$, there exists $u\in S(B_X,x^*,k(\alpha-\gamma))=S(B_X,x^*,a\alpha)$ such that  $\|x-u\|\ge 2-\varepsilon$. From that we get $a\in A$. Analogically we can show that if $B$ is nonempty, then $b=\inf B\in B$.

It is not hard to see that neither $A$ nor $B$ can be empty. Indeed, if $A=\emptyset$ (the case $B=\emptyset$ is analogical), then by assumption $(1,\infty)\subset B$. However, according to the previous argumentation we now get that $1\in B$, i.e. $y$ is a $\Delta$-point, which is a contradiction.
Therefore, $A=[a,\infty)$ and $B=[b,\infty)$. From the assumption we can easily see that $1/a+1/b\ge 1$, hence there exist $p,q>1$ that satisfy $1/p+1/q=1$ such that $p\in A$ and $q\in B$.
\end{proof}
%
%
%
%
%
%
%
%
%
%
%
%
%
%
%
%
%
%
%
\bibliographystyle{amsplain}

\end{document}